\documentclass[11pt]{article}

\usepackage{graphicx,amsmath,amsthm,color}

\newtheorem{thm}{Theorem}[section]
\newtheorem{prop}[thm]{Proposition}

\newtheorem{conj}[thm]{Conjecture}
\newtheorem{fact}[thm]{Fact}

\newcommand{\C}[1]{\ensuremath{\mathbf{C}^{#1}}}
\newcommand{\CP}[1]{\ensuremath{\mathbf{CP}^{#1}}}
\newcommand{\RP}[1]{\ensuremath{\mathbf{RP}^{#1}}}
\newcommand{\R}[1]{\ensuremath{\mathbf{R}^{#1}}}
\newcommand{\RR}{\ensuremath{\mathcal{R}}}
\newcommand{\OO}{\ensuremath{\mathcal{O}}}
\newcommand{\Ss}{\ensuremath{\mathcal{S}}}
\newcommand{\CC}[1]{\ensuremath{\mathcal{C}_{#1}}}
\newcommand{\T}[1]{\ensuremath{\mathcal{T}_{#1}}}
\newcommand{\Alt}[1]{\ensuremath{\mathcal{A}_{#1}}}
\newcommand{\Z}[1]{\ensuremath{\mathbf{Z}_{#1}}}
\newcommand{\I}{\ensuremath{\mathcal{I}}}
\newcommand{\Ih}{\ensuremath{\widehat{\mathcal{I}}}}
\newcommand{\It}{\ensuremath{\widetilde{\mathcal{I}}}}
\newcommand{\Ah}{\ensuremath{\widehat{A}}}
\newcommand{\At}{\ensuremath{\widetilde{A}}}

\title{Dynamics of a soccer ball}
\author{Scott Crass\\
Mathematics Department\\
CSULB\\
Long Beach, CA  90840\\
scrass@csulb.edu}

\begin{document}

\maketitle

\begin{abstract}

Exploiting the symmetry of the regular icosahedron, Peter Doyle and Curt McMullen constructed a solution to the quintic equation.  Their algorithm relied on the dynamics of a certain icosahedral equivariant map for which the icosahedron's twenty face-centers---one of its special orbits---are superattracting periodic points.  The current study considers the question of whether there are icosahedrally symmetric maps with superattracting periodic points at a $60$-point orbit. The investigation leads to the discovery of two maps whose superattracting sets are configurations of points that are respectively related to the soccer ball and a companion structure.  It concludes with a discussion of how a generic $60$-point attractor provides for the extraction of all five of the quintic's roots.

\end{abstract}

\section{Overview}

Felix Klein's work on solving the quintic turned the regular icosahedron into a
mathematical industry. \cite{klein}  Nearly 100 years after Klein developed a
solution, \cite{dm} exploited icosahedral geometry and algebra to
manufacture an algorithm that harnesses symmetrical dynamics to solve the
quintic.    At the core of their method is a degree-$11$ rational map
$f$ on \CP{1}---the Riemann sphere---with two distinctive features: 1) $f$
enjoys the same symmetries as the icosahedron and 2) its $20$ critical points
are $2$-periodic:
$$f^2(a)=f(f(a))=a \quad \text{and} \quad  f(a)\neq a.$$
As a consequence, the critical points are superattracting pairs and the union
of their \emph{basins of attraction} have full measure on the sphere.

Due to the symmetry of the Doyle-McMullen map $f$, the $20$-point attracting
set constitutes an orbit under the icosahedral action on the sphere.  However,
there are $60$ rotational symmetries of the icosahedron---a group isomorphic to
the alternating group \Alt{5}---so that the $20$-point orbit is special.
Indeed, it coincides with the icosahedron's face-centers.

There are two other special icosahedral orbits: $12$ vertices and $30$
edge-midpoints.  For each of these sets, there is an associated map that is
both $2$-periodic and critical exactly on the respective orbits.  A natural
question is whether there are icosahedrally symmetric maps that share
these properties for an orbit consisting of $60$ points.  After laying out the
essential geometric  and algebraic infrastructure, the discussion will conduct a regulated
exploration of the question, one that leads to the discovery of two maps whose geometric properties are associated with the soccer ball.

\section{Icosahedral geometry, invariants, and maps}

The treatment here appears in summary form.  Details can be found in
\cite{klein} and \cite{dm}.

One of the five regular polyhedra, the icoshedron consists of $20$ triangles
five of which meet at each of twelve vertices (Figure~\ref{fig:icos}).
Situating the icosahedron with vertices at $0$ and $\infty$ and a
line of reflective symmetry along the real axis, the remaining ten vertices are:
\begin{align*}
\frac{1\pm \sqrt{5}}{2}\ ,\
\frac{-5+\sqrt{5}\pm i \sqrt{10 (5+\sqrt{5})}}{2(\pm \sqrt{5}-5)}\ , \text{and}\\
\frac{2 (5+\sqrt{5}) \pm i (5-\sqrt{5}) \sqrt{2 (5+\sqrt{5})}}{4(5\pm
\sqrt{5})}
\end{align*}
where all choices of signs are made.

The group of $60$ projective transformations \I\ that act on \CP{1} as
rotational symmetries of the icosahedron is isomorphic to the alternating group
\Alt{5}. Under the icosahedral action, a set of five tetrahedra---disjoint
quadruples of the $20$ face-centers---undergoes the \Alt{5} permutations.
Actually, there are two chiral systems of five tetrahedra which can be paired to
form five cubes.  Along with the vertices and face-centers, the $30$
edge-midpoints complete the special icosahedral orbits:  points that are
fixed by a  member of \I\ other than the identity.

\begin{figure}[ht]

\resizebox{3in}{!}{\includegraphics{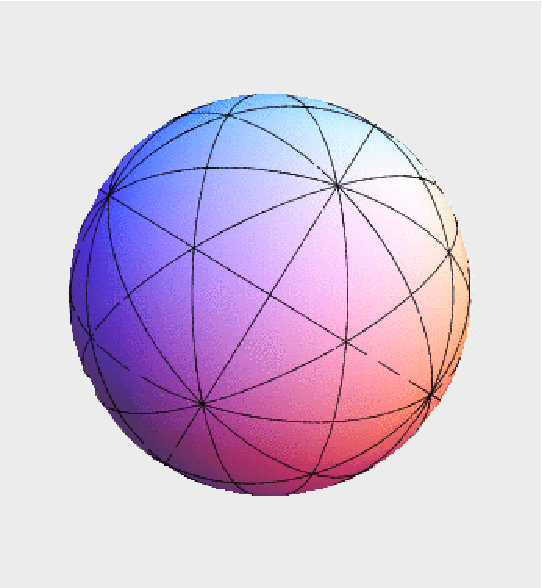}}

\caption{Regular icosahedral structure rendered on the sphere: a configration of
$15$ great circles that act as mirrors of reflective symmetry. Selecting certain
arcs along the circles carves the sphere into either $12$ pentagons, $20$
triangles, or $30$ rhombuses.}

\label{fig:icos}

\end{figure}

We need to address a technical issue here.  The icosahedral action on \CP{1} can
be lifted to a subgroup \Ih\ of $GL_2(\C{})$ with order $240$ and the structure
of a \emph{complex reflection group}; \Ih\ is generated by $30$ order-$2$
complex reflections on \C{2}.  (See \cite{st}.)  Taking
$$\pi:GL_2(\C{})\longrightarrow PGL_2(\C{})$$
to be the standard projection
$$\pi(\Ah)=\bigl\{\alpha\,\Ah\ |\ \alpha\in \C{}-\{0\}\bigr\},$$
each projective fiber $\pi|_{\Ih}^{-1}(A)$ of the restriction of $\pi$ to \Ih\ contains four linear maps
$$i^k \Ah\ \text{for}\ k=0,\dots,3\ \text{and where}\ \det{\Ah}=1.$$
Of course, an \I-orbit in \CP{1} of size $n$ lifts to an \Ih-orbit in \C{2} with
$4n$ elements.  In the treatment that follows, we work in the lifted setting.

Associated with the twelve-point \I-orbit $\OO_{12}$ in \CP{1} are twelve linear forms
$$L_k=a_k\,x+b_k\,y \qquad k=1,\dots,12$$
that define the members of $\OO_{12}$ as lines in \C{2}.  Moreover, we can choose the $L_k$ so that, up to multiplicative character, \Ih\ permutes the forms.  This means that for all $\Ah\in \Ih$,
$$L_k\circ \Ah=i^\ell\,L_{k'}$$
where $\ell$ depends on \Ah.  However, when an element of \Ih\ is applied to the degree-$12$ polynomial
$$F(x,y)=\prod_{k=1}^{12} L_k=x y \bigl(x^{10}-11\,x^5 y^5-y^{10}\bigr),$$
the characters cancel and $F$ is \Ih-invariant.  That is,
$$F(\Ah(x,y))=F(x,y)\qquad \text{for all $\Ah\in \Ih$}.$$
Alternatively, we construct an invariant in affine terms:
$$F(z)=\prod_{k=1}^{11} (z-\lambda_k)$$
where $z=\dfrac{x}{y}$ is an affine coordinate on \C{}, and
$\lambda_k=-\dfrac{b_k}{a_k}$
are the vertices of the icosahedron.  Note that this gives a degree-$11$
polynomial since one of the vertices is $\infty$ (taken to be $\lambda_{12}$).

From classical invariant theory, we obtain a second \Ih-invariant
$$H(x,y)=\text{hess}(F(x,y))=\ x^{20}+228\,x^{15} y^5+494\,x^{10}y^{10}-228\,x^5 y^{15}+y^{20}$$
where
$$
\text{hess}(F(x,y))=
\det{\begin{pmatrix}F_{xx}&F_{xy}\\F_{xy}&F_{yy}\end{pmatrix}}
$$
and subscripted variables indicate partial differentiation .
Furthermore, reflection group theory tells us that $F$ and $H$ are algebraically
independent generators of the ring of \Ih-invariant forms $\C{}[x,y]^{\Ih}$.
Note that the icosahedral face-centers appear as $\{H=0\}$.

A \emph{relative \Ih-invariant} arises from the determinant of the Jacobian of
$F$ and $H$ relative to $x,y$:
\begin{align*}
T(x,y)&=\ J_{(F,H)}=\ \det{\begin{pmatrix}F_x&F_y\\H_x&H_y\end{pmatrix}}\\
&=\ x^{30}-522\,x^{25} y^5-10005\,x^{20}y^{10}-10005\,x^{10} y^{20}+522\,x^5 y^{25}+y^{30}.
\end{align*}
Relative invariance means that a multiplicative character appears when \Ih\ acts on $T$.  Specifically,
$$T\circ \Ah=\pm T$$
where the plus sign occurs for the determinant-$1$ subgroup \It\ of \Ih\ whose order
is $120$.

While $F$ and $T$ are algebraically independent, the three invariants satisfy a
relation in degree $60$:
$$1728\,F^5-H^3+T^2=0.$$
We can now take up the matter of icosahedrally symmetric maps.  The treatment is simpler in the setting of \It.

A map $f:\C{2}\rightarrow\C{2}$ is \emph{\It-equivariant} if it preserves
group-orbit structure; that is, $f$ sends an \It-orbit to an \It-orbit.
Expressed
algebraically,
$$f\circ \At=\At\circ f \qquad \text{for all $\At \in \It$.}$$
Classical invariant theory provides for the construction of equivariants from
invariants.  In the present case, the relevant cross-product-like procedure is
as follows.
\begin{prop}

If $G(x,y)$ is \It-invariant,
$$
\times G:=\bigl(G_y,-G_x\bigr)
$$
is a relative \It-equivariant.

\end{prop}
Applied to the \emph{basic invariants}:
\begin{align*}\nonumber
\phi=\ \times F=&\ \bigl(-x^{11}+66\,x^6 y^5+11\,x y^{10},
11\,x^{10} y-66\,x^5 y^6-y^{11}\bigr)\\ \nonumber
\eta=\ \times H=&\ \bigl(-57\,x^{15} y^4-247\,x^{10}y^9+171\,x^5
y^{14}-y^{19},\\
&\ x^{19}+171\,x^{14} y^5+247\,x^9 y^{10}-57\,x^4 y^{15}\bigr)
\end{align*}
These two maps are basic in that they generate the module of \It-equivariants
over \It-invariants.

Passing to the projective line, the basic maps also admit a geometric description. In the case of $\phi$, each triangular
dodecahedral face $\mathcal{F}$ wraps and twists canonically around the dodecahedral structure in such a way that
$$\phi(\mathcal{F})=\CP{1}-\mathcal{F}_a^\circ$$
where $\mathcal{F}_a^\circ$ is the interior of the face $\mathcal{F}_a$ antipodal to $\mathcal{F}$.  Canonical behavior means that each edge is,  as a set, sent to the antipodal edge while vertices and edge-midpoints map to their antipodes.  By symmetry, face-centers are fixed points.  Around each vertex are three interior pentagonal sectors each of which $\phi$ opens up onto two of the pentagonal sectors at the antipodal vertex.  Accordingly, the $20$ dodecahedral vertices are critical and period-$2$.  Moreover, since the degree of the critical polynomial $J_\phi$ is $20=2\cdot 11-2$, the vertices exhaust the critical set.  This strong kind of \emph{critical finiteness} produces a map in which the basins of attraction at the superattracting vertices have full measure on the sphere.  By considering faces of the icosahedron, analogous phenomena occur in the case of $\eta$ under which the $12$ vertices have a critical multiplicity of three.

But, you might wonder: What about the identity map?  It must be
\It-equivariant since it commutes with everything.  The answer lies in the relation on the basic invariants which manifests itself in the regime of maps by considering the degree-$31$ family of maps that are equivariant under \It:
$$g=a\,T\,\epsilon+b\,H\,\phi+c\,F\,\eta.$$
Here, $\epsilon$ is the identity map, which is projectively equivalent to
$T\,\epsilon$. By elementary invariant theory, the Jacobian determinant $J_f$
of an \It-equivariant $f$ is an \It-invariant so that, for the family of
``$31$-maps,'' $J_f$ is a $3$-parameter form of degree $60$.  However, due to
the relation satisfied by $F$, $H$, and $T$, the family of degree-$60$
\It-invariants,
$$\ell\,F^5+m\,H^3+n\,T^2,$$
is two-dimensional.  Thus, the maps in $g$ must also satisfy a relation, the
explicit form of which is $$5\,T\,\epsilon+5\,H\,\phi-3\,F\,\eta=(0,0).$$
Hence, $\phi$ and $\eta$ produce the map $T\epsilon$ that gives the identity on
\CP{1}.

Since the degree of the critical polynomial $J_\phi$ is $20$,
the critical set $\CC{\phi}=\{J_\phi=0\}$ consists of the ``$20$-points''
$\{H=0\}$.  (Note: when describing orbits, we refer to points in projective
space, that is, the sphere.)  By similar considerations, $\CC{\eta}=3\cdot
\{F=0\}=3\cdot \{\text{$12$-points}\}$ where the coefficient indicates
three-fold
multiplicity at each critical point.

For any group action $G$, when a point $p$ is fixed by some member $A$ of $G$,
the image of $p$ under a holomorphic $G$-equivariant is also fixed by $A$.  In
the icosahedral case, the only two possibilities for the equivariant image of a
point in a special orbit are to be fixed or exchanged with its antipode.  In
the latter case, the point has period two.  Table~\ref{tab:spec_pts} collects
the actual behavior of the basic equivariants on the special orbits.

\begin{table}[ht]
\begin{tabular}{c|c|c|c}
orbit size $\rightarrow$&$12$&$20$&$30$\\
map $\downarrow$&&&\\
\hline
$\phi$&fixed&period-$2$&period-$2$\\
$\eta$&period-$2$&fixed&period-$2$
\end{tabular}
\caption{Special orbits under basic \I-equivariants}
\label{tab:spec_pts}
\end{table}

Combining criticality and periodicity yields, in each case, a critically-finite map---the orbit of each critical point is
finite---for which the critical $2$-cycles are superattracting.  In such a circumstance, the theory of one-dimensional complex dynamics informs us that the basins of attraction for the critical cycles are dense---indeed, have full
measure---on the sphere.

\section{Icosahedral maps with sixty critical points}

The task is to construct maps with periodic critical points residing in
a $60$-point orbit.  A critical polynomial with $60$ roots requires a map with algebraic degree $31$.  Our
approach will be to treat a point in the critical set as a parameter to be determined by imposing conditions on the family of $31$-maps.  Results of \emph{Mathematica} computations are labeled ``Facts."  (Files that establish the facts are available at \cite{web}.)

Let
$$f_{(a,b)}=a\,H\,\phi+b\,F\,\eta$$
be the family of $31$-maps.  We want to determine coefficients $(A,B)$ and
points $p$ that satisfy the conditions:
\begin{enumerate}

\item $p$ belongs to a $60$-point orbit $\I(p) \subset \CP{1}$

\item $\I(p)$ is the critical set of $f_{(A,B)}$

\item $f_{(A,B)}(\I(p))=\I(p)$ (implies that $p$ is periodic).

\end{enumerate}

We can readily identify three $31$-maps that degenerate to maps on \CP{1} of
lower degree, namely,
$$f_{(1,0)}=H\phi \qquad f_{(0,1)}=F\eta \qquad f_{(5,-3)}=T\epsilon.$$
In each of these cases, one of the special icosahedral orbits constitutes the
critical set with the appropriate algebraic multiplicity: three, five, and two
respectively.

\subsection{A specific approach}

Finding all maps that satisfy the desired conditions is a formidable
computational task. Accordingly, the search will be subject to a further
restriction:
$$
(4)\  f_{(A,B)}\ \text{exchanges $p$ with its antipode $\widetilde{p}$ and
hence, $p$ is $2$-periodic}.
$$

Working in \C{2}, let $p=(z,1)$ so that
$\widetilde{p}=(1,-\overline{z})$.  Requiring that $f_{(a,b)}(p)=\widetilde{p}$
determines $a$ and $b$ as functions of $z$ and $\overline{z}$.
When $p$ is forced to be critical by setting
$$J(z):=J_{f_{(a,b)}}(p)=0,$$
the polynomial $J(z)$ is divisible by $F(z)$, $H(z)$, and $T(z)$.  This outcome
follows from the observation above that the maps $H\phi$, $F\eta$, and
$T\epsilon$ along with the appropriate special orbit each satisfy the ``critical
equation.''

\begin{fact}

The result of dividing $J$ by the product $FHT$ is
\begin{align*}
M(z)=& z^{60} \overline{z}-285 z^{56} \overline{z}^2+2820 z^{55}
\overline{z}+3410
   z^{54}+3800 z^{51} \overline{z}^2-37794 z^{50} \overline{z}\\
&+83700 z^{49}-3799050 z^{46} \overline{z}^2+7302980 z^{45}\overline{z}+12227175
z^{44}\\
&-15405600 z^{41} \overline{z}^2+23401665 z^{40}
   \overline{z}+15580600 z^{39}\\
&-143079850 z^{36} \overline{z}^2+80976024 z^{35} \overline{z}+168348600
z^{34}\\
&+203866260 z^{31} \overline{z}^2-203866260 z^{29}+168348600 z^{26}
\overline{z}^2\\
&+80976024 z^{25} \overline{z}-143079850 z^{24}-15580600 z^{21}
\overline{z}^2\\
&-23401665 z^{20} \overline{z}+15405600 z^{19}+12227175
   z^{16} \overline{z}^2+7302980 z^{15} \overline{z}\\
&-3799050 z^{14}-83700 z^{11} \overline{z}^2+37794 z^{10} \overline{z}-3800
   z^9+3410 z^6 \overline{z}^2\\
&+2820 z^5 \overline{z}-285 z^4-\overline{z}+1.
\end{align*}

\end{fact}
Alternatively, we can produce special $31$-maps by solving over the reals the
equations
$$R(u,v)=0 \qquad S(u,v)=0 $$
where  $z=u+i\,v$ and $M=R+i\,S$.  The high degree involved thwarts efforts to
bring techniques of algebraic
geometry to bear on the problem.

We can, however, establish a geometric property that a solution must possess.
Indeed,
this property follows from a consequence of conditions (3) and (4) that we
demand of a special
$31$-map: $\widetilde{p}\in \I(p)$.

First, an observation is  appropriate.  The holomorphic icosahedral
action \I\ admits a degree-$2$ extension to a group $\I\, \cup\, \alpha\,\I$
where
$\alpha$ is the antipodal map.  The antiholomorphic coset $\alpha\,\I$ contains
reflections through $15$ great circles---the icosahedron's mirrors of reflective
symmetry as illustrated in Figure~\ref{fig:icos}.

\begin{thm}
Suppose $p\in\CP{1}$ with $\alpha(p):=\widetilde{p}\in \I(p)$.  Then $p$
(and $\widetilde{p}$) belong to a great circle of reflective icosahedral
symmetry.
\end{thm}

\begin{proof}
Since $\widetilde{p}$ belongs to the \I-orbit of $p$, there is an $A\in \I$
with $A(p)=\widetilde{p}$.  The only element of the rotational icosahedral
group that moves a point to its antipode is a half-turn about an axis through
antipodal edge-midpoints.  Furthermore, $A(\widetilde{p})=p$.  Indeed, $A$ acts
as the antipodal map on a great circle that functions as a mirror $\mathcal{M}$
under the extended
icosahedral group.  Moreover, each hemisphere determined by $\mathcal{M}$
is preserved by $A$ so that for any point $q$ away from $\mathcal{M}$,
$A(q)\neq \widetilde{q}$.  Hence, $p,\widetilde{p}\in \mathcal{M}$.
\end{proof}
In the chosen coordinates, the real axis \RR---as an \RP{1}---is a great circle
of reflective icosahedral symmetry.  For any solution $p$ to $R=S=0$, every
member of $p$'s orbit $\I(p)$ is also a solution.  Since $p$ belongs to an
icosahedral mirror and \I\ is transitive on mirrors, some element in $\I(p)$
lies on the real axis.

\begin{fact}

We can find all $31$-maps with the desired
properties by solving over the real numbers the equation
\begin{align*}
 M|_{z=\overline{z}}=&\ z \bigl(z^2+1\bigr)
   \bigl(z^2-z-1\bigr) \bigl(z^4-2 z^3-6
   z^2+2 z+1\bigr)\\
& \bigl(z^4+3 z^3-z^2-3
   z+1\bigr) \bigl(z^{48}+14 z^{46}-280
   z^{45}+161 z^{44}\\
&-1039 z^{43}+364
   z^{42}-666 z^{41}+621 z^{40}+27291
   z^{39}-32823 z^{38}\\
&+394034
   z^{37}-241717 z^{36}+557621
   z^{35}-499383 z^{34}+392493
   z^{33}\\
&+478383 z^{32}-854138
   z^{31}+1147057 z^{30}-3389037
   z^{29}\\
&+3865560 z^{28}+1191562
   z^{27}+5421980 z^{26}+744833
   z^{25}\\
&+10215020 z^{24}-744833
   z^{23}+5421980 z^{22}-1191562
   z^{21}\\
&+3865560 z^{20}+3389037
   z^{19}+1147057 z^{18}+854138
   z^{17}\\
&+478383 z^{16}-392493
   z^{15}-499383 z^{14}-557621
   z^{13}-241717 z^{12}\\
&-394034
   z^{11}-32823 z^{10}-27291 z^9+621
   z^8+666 z^7+364 z^6\\
&+1039 z^5+161
   z^4+280 z^3+14 z^2+1\bigr).
\end{align*}

\end{fact}
For numerical solutions to this degreee-$61$ equation, \emph{Mathematica} gives
$61$ complex values, $19$ of which are real.  Since the stabilizer $\I^{\RR}$ of
\RR\ in \I\ is a Klein-$4$ action, an $\I^{\RR}$-orbit consists of two or
four points.  Some of the real solutions occur at special \I-orbits.
Three of the points are vertices of the icosahedron. A fourth vertex
belongs to the real axis as a point at infinity.  Among the $16$ remaining real
solutions are four face-centers and four edge-midpoints the latter of which
decomposes into two $\I^{\RR}$-orbits.  Note that the special orbits occur at
intersections of icosahedral mirrors.  The remaining eight solutions on \RR\
decompose into two $4$-point $\I^{\RR}$-orbits and, thereby, indicate the
presence of two $31$-maps with $60$ critical $2$-periodic points.
 Let $\OO_1$ and $\OO_2$ denote the respective $60$-point \I-orbits.

Graphical evidence that there are only two such maps appears by application of
Newton's method on \R{2} to the polynomials $R$ and $S$.  In
Figure~\ref{fig:newt_meth}, the Newton map $n$ is iterated on a
``triangle'' of initial conditions that is a fundamental domain for the action
\I.  A colored region indicates a basin of attraction associated with
an attracting  point under $n$.  The triangle is carved into $\binom{1000}{2}$
cells and the center $(x_0,y_0)$ of each cell is the starting point for a
trajectory $\bigl(n^k(x_0,y_0)\bigr)_{k=0}^{\infty}$. If the trajectory
converges to one of the five known solutions to $R=S=0$ in fewer than $200$
iterations, the initial cell is given a color referenced to the attracting
point. Recall that such a solution is an \I-orbit.  The experimental outcome is
that the vast majority of initial conditions have a trajectory that converges to
a
known solution in the allotted number of iterations.  The three special orbits
correspond to red (vertices), green (face-centers), and blue (edge-midpoints)
while the basins of $\OO_1$ and $\OO_2$ are olive and violet.  Something to
observe is that, due to the choice
of $\mathcal{T}$, some of the attracting points are not on the real axis, but
belong to one of the other mirrors.

\begin{fact}

Recalling the parameters $(a,b)$ that distinguish icosahedral $31$-maps, any
point in the orbit $\OO_1$ yields values $(A,B)\approx (1,1.5954)$ and hence, a
map:
\begin{align*}
g=&\ f_{(A,B)}\approx \bigl(x^{31}+1980.7608 x^{26}
   y^5-26690.072 x^{21} y^{10}-129309.31
   x^{16} y^{15}\\
&+61784.718 x^{11}
   y^{20}+7547.2935 x^6 y^{25}-42.908084 x
   y^{30},\\
&-42.908084 x^{30} y-7547.2935
   x^{25} y^6+61784.718 x^{20}
   y^{11}\\
&+129309.31 x^{15}
   y^{16}-26690.072 x^{10}
   y^{21}-1980.7608 x^5
   y^{26}+y^{31}\bigr).
\end{align*}
In the case of $\OO_2$, we get $(A,B)\approx (1,.0280899)$ and
\begin{align*}
h=&\ f_{(A,B)}\approx \bigl(x^{31}+194.02245 x^{26}
   y^5-14778.483 x^{21} y^{10}-36994.493
   x^{16} y^{15}\\
& +10533.539 x^{11}
   y^{20}+2531.8876 x^6 y^{25}-11.561797 x
   y^{30},\\
&-11.561797 x^{30} y-2531.8876
   x^{25} y^6+10533.539 x^{20}
   y^{11}\\
& +36994.493 x^{15}
   y^{16}-14778.483 x^{10}
   y^{21}-194.02245 x^5
   y^{26}+y^{31}\bigr).
\end{align*}

\end{fact}

\begin{conj}
The only icosahedrally symmetric $31$-maps with a $60$-point critical set whose
members are exchanged with their antipodes are $g$ and $h$.
\end{conj}

\begin{figure}[ht]
\resizebox{2.2in}{!}{\includegraphics[]{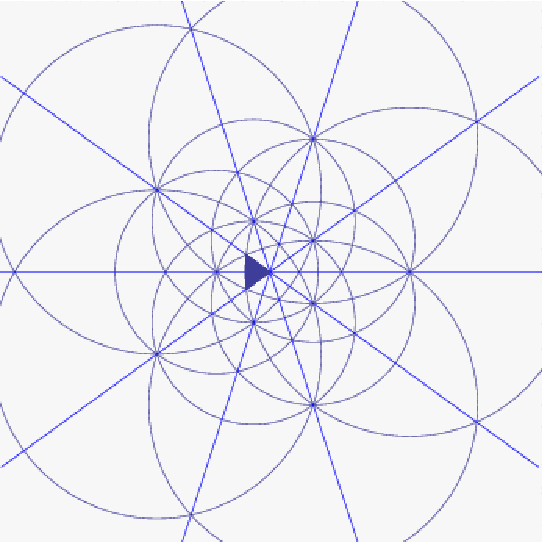}} \hspace*{.1in}
\resizebox{2.2in}{!}{ \includegraphics[]{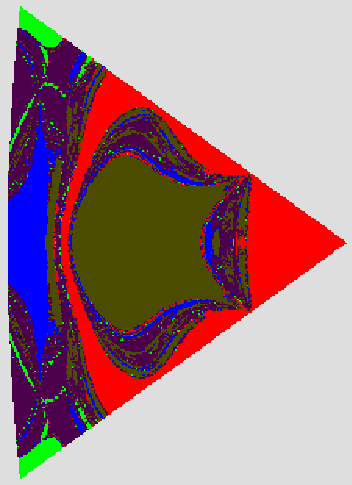}}
\caption{\emph{Left}: fundamental triangle $\mathcal{T}$ in the planar
icosahedral net.
\emph{Right}: basins of attraction on $\mathcal{T}$ for Newton's method relative
to $R$
and $S$.}
\label{fig:newt_meth}
\end{figure}

\subsection{A general approach}

We can search for maps with forward invariant critical sets that consist of
$60$ points in an another way.  Working with $31$-maps in coordinates on the
quotient space $\CP{1}/\I$, Peter Doyle found a polynomial $P$ with integer
coefficients each of whose roots corresponds to a map with the desired
properties. In particular, a certain degree-$12$ factor $Q$ of $P$ has two real
roots that result in the soccer ball maps $g$ and $h$:
\small
\begin{align*}
 Q=&\ 82308000000 w_1^{12}+570668800000 w_2 w_1^{11}+1850218100000 w_2^2
w_1^{10}\\
&+2095583560000 w_2^3 w_1^9-1409601459875 w_2^4 w_1^8-6610001558770 w_2^5
 w_1^7\\
&-8497318777952 w_2^6 w_1^6-6132369876696 w_2^7 w_1^5-2739682103040 w_2^8
 w_1^4\\
&-746454705120 w_2^9 w_1^3-110441517120 w_2^{10} w_1^2-5598533376 w_2^{11}
w_1\\
&+262766592 w_2^{12}
\end{align*}

\normalsize
He was able to show that $Q$'s Galois group is the symmetric group
$S_{12}$ so that we cannot hope to express the roots as well as the maps in
terms of radicals.

Obtaining the condition $P=0$ opens up the current project to intriguing
possibilities.  What other maps arise from the roots pf $P$?  Does each map
correspond to an icosahedrally-symmetric polyhedron such as the snub
dodecahedron?  Future research will explore these issues at length.

\subsection{Soccer ball dynamics}

To pursue the behavior of $g$ and $h$, we ask about the configuration of their
critical points.  Plotting the elements in
$$\OO_1=\{J_g=0\}\ \text{and}\ \OO_2=\{J_h=0\}$$
on the sphere (Figure~\ref{fig:crit_pts}) reveals an underlying structure to the
respective maps.  In one case, the orbit appears as vertices of a soccer
ball---a truncation of the icosahedron at the vertices. Points in the other
orbit are arranged as vertices of what we might call a ``dual'' soccer ball---a
polyhedron with $20$ triangular and $12$ decagonal faces produced by
truncating at the vertices of the dodecahedron, the icosahedron's dual.

\begin{figure}[ht]
\resizebox{2.2in}{!}{\includegraphics[]{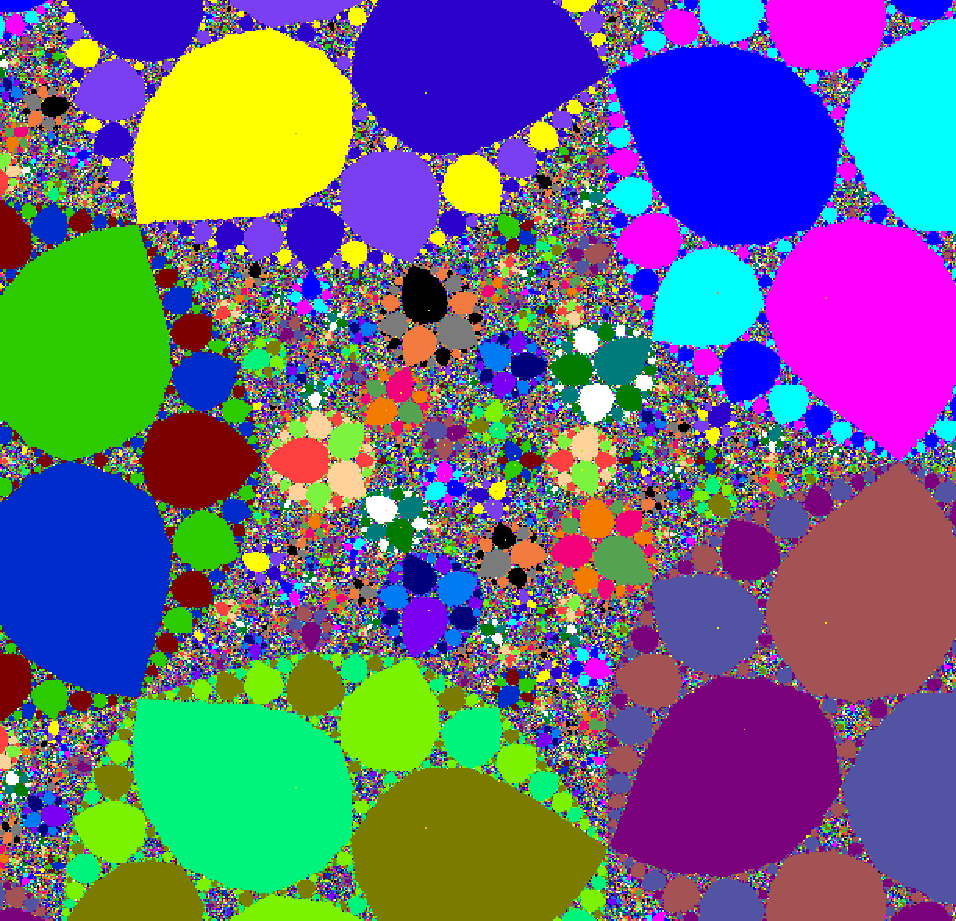}} \hspace*{.1in}
\resizebox{2.22in}{!}{ \includegraphics[]{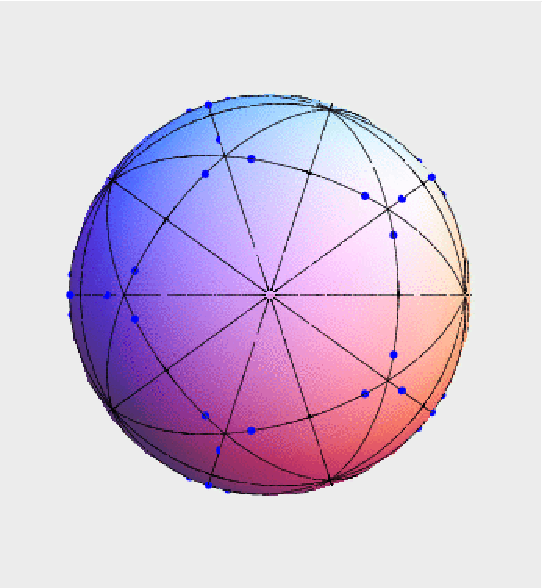}}
\caption{\emph{Left}: $60$ critical points of $g$ as vertices of a soccer ball.
\emph{Right}: $60$ critical points of $h$ as vertices of a dual soccer ball.}
\label{fig:crit_pts}
\end{figure}

If we assign a color to each superattracting $2$-cycle of soccer ball or dual soccer ball vertices and plot the basin of attraction under $g$ and $h$ in the appropriate color, the results rendered in the plane by \emph{Dynamics 2} appear in Figure~\ref{fig:basins}.  \cite{dyn2}  Inspection reveals a basin for each point appearing in Figure~\ref{fig:crit_pts}.

\begin{figure}[h]
\resizebox{2.2in}{!}{\includegraphics[]{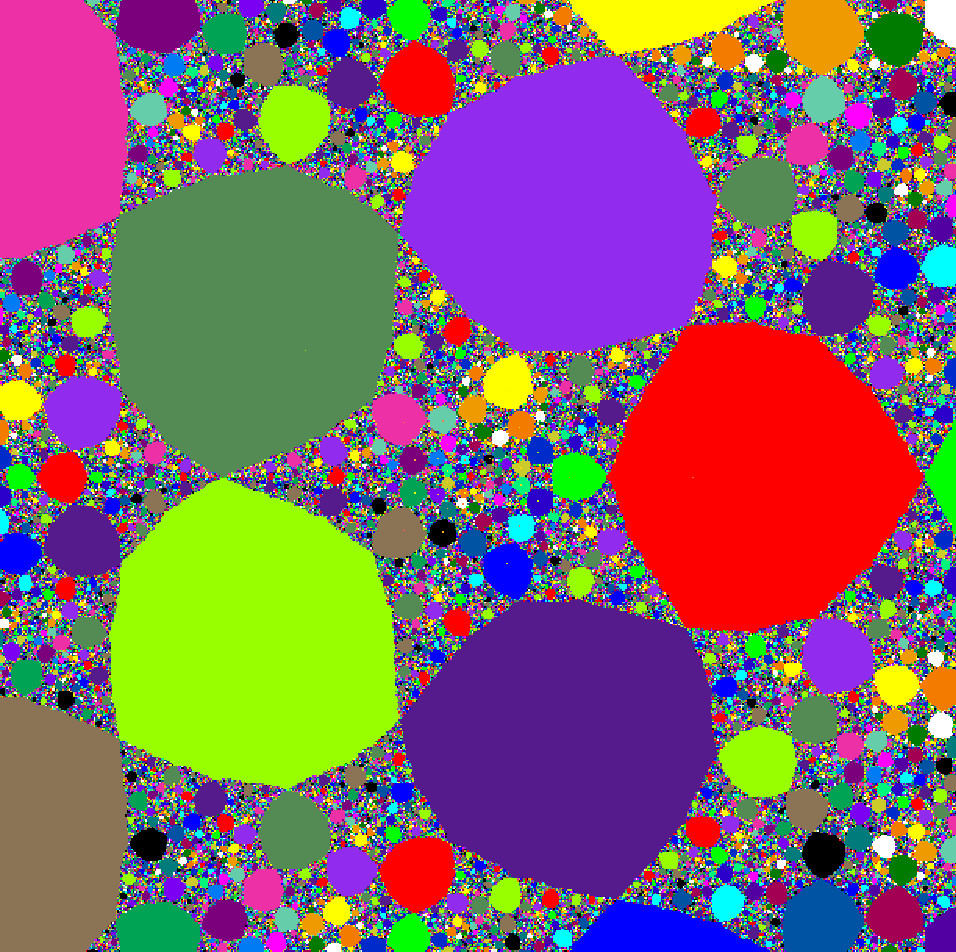}} \hspace*{.1in}
\resizebox{2.2in}{!}{ \includegraphics[]{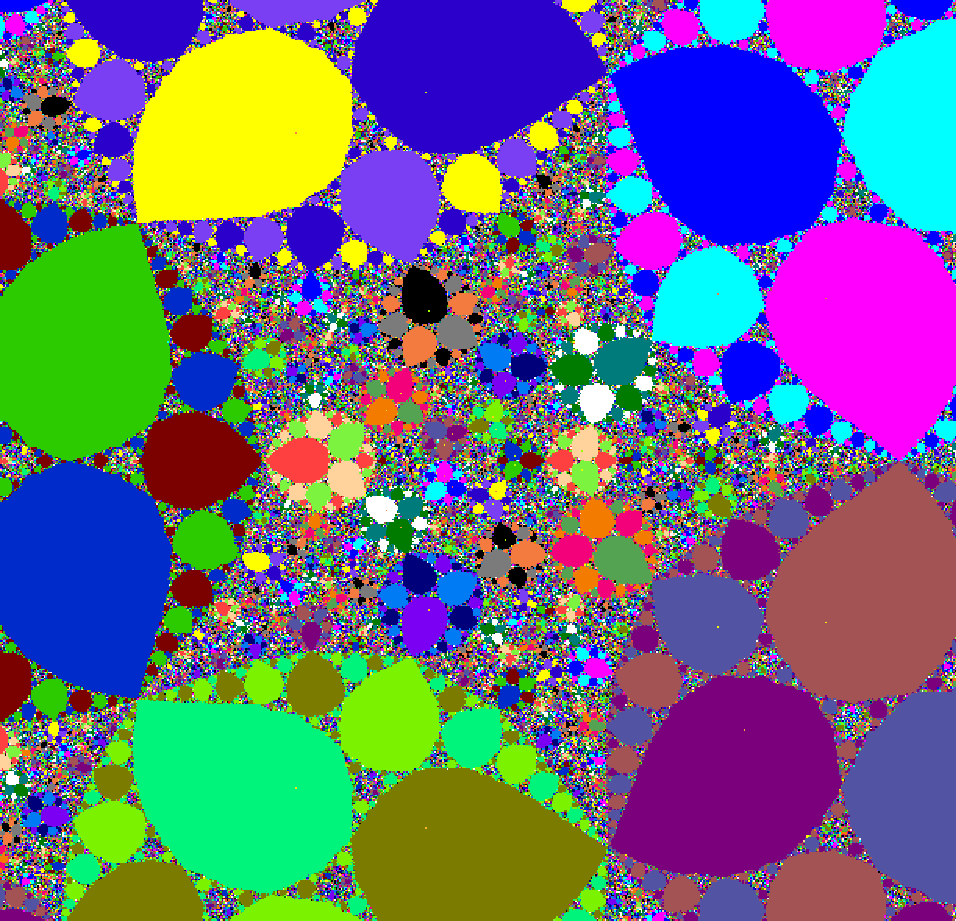}}
\caption{\emph{Left}: $30$ basins of attraction for soccer ball map.
\emph{Right}: $30$ basins of attraction for dual soccer ball map}
\label{fig:basins}
\end{figure}

For the remainder of the discussion, we will consider the soccer ball map $g$
and leave to the interested reader the analogous treatment with respect to $h$.

To complete the soccer ball structure we need to determine edges that
are dynamically meaningful.  Geometrically, the edges are boundaries between
either two hexagonal faces or a hexagonal and pentagonal face.   Each
hexagon-hexagon edge lies on an icosahedral mirror.  Since $g$ is symmetric
under the extended group $\I\,\cup\,\alpha\,\I$ and a mirror $\mathcal{M}$ is
pointwise fixed by an element of $\alpha\,\I$, $g(\mathcal{M})=\mathcal{M}$.
Accordingly, the segment on $\mathcal{M}$ between two vertices in ``adjacent''
pentagons provides a natural hexagon-hexagon edge.

A more interesting state of affairs obtains in the case of pentagon-hexagon
edges.  But, before turning to the question of how to specify a pentagon-hexagon
edge, we should consider what there is to gain from such an accomplishment.

Recalling the geometric account of the dodecahedral and icosahedral maps of
degrees $11$ and $19$, the soccer ball's $32$ faces are tantalizing.  Can we
describe the behavior of $g$ in a geometric way, that is, as a ``wrap-twist" of
one face onto $31$ other faces?  Specifically, a pentagon or hexagon would map
to the complement of the antipodal polygon with vertices and edge-midpoints sent
to their antipodes while face-centers are fixed.  The degree of such a map would
be $31$ and antipodal vertices would form critical $2$-cycles at which the local
behavior is squaring---precisely properties that $g$ possesses.

Now, consider a pentagon realized by five soccer ball vertices about an
icosahedral vertex $C$---the pentagon's center.  Since the only symmetry in \I\
that moves one of the pentagonal vertices $X$ to an adjacent vertex $Y$ is a
$\frac{1}{5}$-turn about $C$, a curve between $X$ and $Y$ that does not contain
another pentagonal vertex cannot be setwise invariant under some element in \I.
Hence, there is no natural ``geometric'' edge between a pentagon and adjacent
hexagon.  We can, however, generate such an edge dynamically.

The idea is to locate a point $Z$ on the icosahedral mirror between $X$ and $Y$ that has period $2$ under $g$
and so, is fixed by $g^2$.  In Figure~\ref{fig:crit_pts}, we can take the central point with $5$-fold rotational symmetry as $0$ and the horizontal great circle as the real axis on which $Z$ resides.

\begin{fact}

$Z\approx .143827.$

\end{fact}
Furthermore, $Z$ is a repelling fixed point under $g^2$.  Now, let $\mathcal{S}$
be an arbitrarily small line segment through $Z$ perpendicular to the real axis.
 Since $g$ is conformal at $Z$ and preserves the real axis, $g^2(\mathcal{S})$
contains $Z$ and is perpendicular to the real axis.

\begin{fact}

Upon iteration, the
trajectory $(g^{2k}(W))$ of a point $W$ on $S$ other than $Z$ tends to one of
the two nearest soccer ball vertices---which are fixed by $g^2$---depending upon the
sign of the imaginary part of $W$.

\end{fact}
However, the trajectory
$$\Sigma=\bigcup_{k=0}^\infty g^k(\mathcal{S})$$
of $\mathcal{S}$, which spans between $X$ and $Y$, is not the pentagon's edge.
After all, $\Sigma$ depends on the chosen length of $\mathcal{S}$.  A further
rationale stems from the fact that $\Sigma$ is an analytic curve. If $\Sigma$
were an edge of a pentagon, $g$ could be extended from the interior of a
pentagon by Schwarz reflection across $\Sigma$.  Evidently, the loss of analyticity occurs at $Z$.    The result would be a pentagon, rather than a hexagon, on the other side of $\Sigma$. So, the $g$-invariant pentagon-hexagon edge $\Pi$ appears as the limit of analytic curves as the length of $\mathcal{S}$ vanishes:
$$\Pi=\lim_{|\mathcal{S}|\to 0} {\bigcup_{k=0}^\infty g^k(\mathcal{S})}.$$
Empirical considerations indicate that $\Pi$ is very nearly the circular arc through $X$, $Y$, and $Z$.  Figure~\ref{fig:pent_edge} depicts on the sphere an approximation of $g$'s
Julia set along with the pentagon-hexagon edges.  The image also depicts hexagon-hexagon edges, arcs of great circles that are invariant under $g^2$ and whose midpoint is fixed and repelling.  Note that the  Julia set is the complement of $g$'s basins of attraction.
\begin{figure}[ht]
\resizebox{4in}{!}{\includegraphics[]{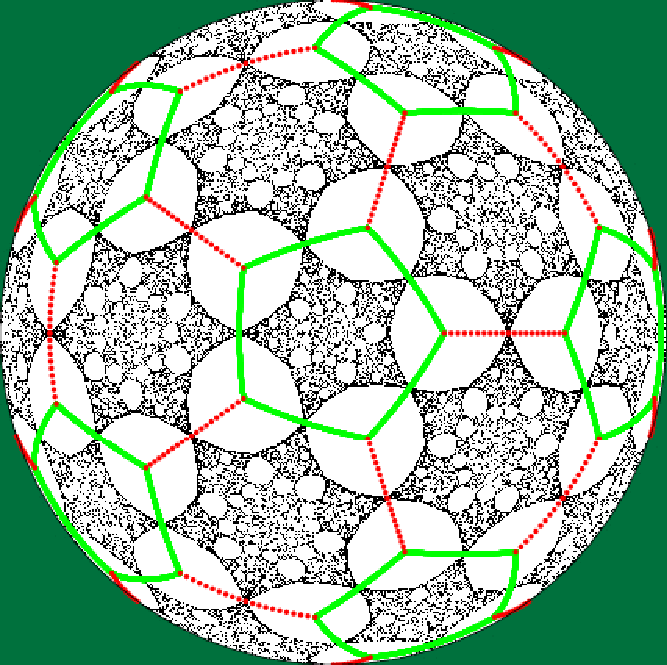}}
\caption{Soccer ball's edges on the Julia set of $g$ (\emph{Mathematica} image created by
Peter Doyle).}
\label{fig:pent_edge}
\end{figure}

For a final observation, consider the $3$-fold intersection of edges at a
vertex $v$---two pentagon-pentagon and one pentagon-hexagon.  Since the local
behavior of $g$ at $v$ is squaring, the three edges trisect a full turn.  Thus, our
dynamical soccer ball realizes \emph{equal polygonal angles} rather than equal
side-lengths and can be approximated by pentagons that are regular and hexagons
that are not.  In this light, we can see why there are only two special $31$-maps: there
are two truncations of the icosahedron that produce polyhedra with $32$ faces
and trisected angles at each of $60$ vertices.  Indeed, we might wonder about
higher-degree maps that are geometrically associated with ployhedra (chiral?)
that are icosahedrally derived. 
\section{Completely solving the quintic}

Determining all of the solutions to an \Alt{5}-symmetric quintic requires that
we  break the symmetry by distinguishing any one of the $60$ orderings of
the roots from the others.  In the setting realized by the icosahedron, \Alt{5}
symmetry is broken by randomly selecting a point $z_0\in \CP{1}$.  Such a point
almost always belongs to a $60$-point \I-orbit.   Since the basins of the superattracting pairs of antipodal vertices have full measure, the trajectory of $z_0$ under $g$ locates such a pair of vertices on the special soccer ball.  Selecting every other iterate, we arrive at a single vertex:
$$z_{\infty}=\lim_{k\to \infty}{g^{2\,k}(z_0)}.$$
(A similar discussion goes through in the case of $h$ and the special dual
soccer ball.)  Since $\I(z_{\infty})$  contains $60$ points, $g$'s dynamics
maintains the full \Alt{5} symmetry-breaking of the initial choice $z_0$.
Accordingly, one iterative run from $z_0$ to $z_{\infty}$ suffices for the
extraction of all five roots.

The algorithm that harnesses a symmetrical map to
a polynomial's solution is explicitly described in \cite{sextic}.  Here, I will
sketch how a complete solution issues from a single data point, namely, one of
the soccer ball's vertices.  The key step is to associate each quintic with
an icosahedral group and, thereby, with a soccer ball map.

\subsection{Quintic resolvent}

For purposes of exposition, we use inhomogeneous coordinates and refer to
projective objects rather than work in the more cumbersome, strictly correct lifted
setting.

As indicated earlier, the $20$-point \I-orbit $\OO_{20}$  of dodecahedral
vertices decomposes into five sets of four tetrahedral orbits
$$\{t_{a1},t_{a2},t_{a3},t_{a4}\}\qquad a=1,\dots,5$$
under respective tetrahedral subgroups \T{a}.  Let
$$T_{a}(z)=\prod_{k=1}^4(z-t_{ak})$$
so that the $T_a$ are degree-$4$ invariants under \T{a}.  Moreover, the $T_a$
undergo \Alt{5} permutation by the action of \I.

Now, form the quintic whose roots are the $T_a(z)$:
$$P_5(s)=\prod_{a=1}^5 \bigl(s-T_a(z)\bigr)=\sum_{k=0}^5 a_k(z)\,s^{5-k}.$$
Since the $T_a$ are permuted by \I, the coefficients $a_k(z)$ are \I-invariants
of degrees $4\,k$ for $k=0,\dots 5$.  The absence of \I-invariants in
degrees $4$, $8$, and $16$ forces
$$a_1=a_2=a_4=0$$
while degree considerations bear upon the remaining coefficients:
$$a_3=b\,F(z),\ a_5=c\,H(z),\ \text{where $a,b$ are constants}.$$
After a change of variable
$$s\to \frac{H^2}{F^3} s,$$
the resulting equation can be expressed as
$$P_{5_Z}(s)=s^5+b\,Z^2\,s^2+c\,Z^3=0$$
where
$$Z=\frac{F^5}{H^3}$$
is an ``icosahedral function'' whose values give \I-orbits.

\subsection{Root extraction}

As developed in \cite{sextic}, selecting a value for the parameter $Z$ amounts to choosing both a quintic $P_{5_Z}$ and an icosahedral action $\I_Z(w)$ where $u=B_{Z}(w)$ is a parametrized change of coordinates and $u$ is a reference coordinate that replaces $z$ as employed previously.  The group $\I_Z(w)$ gives rise to a function $J_{5_Z}$ that extracts a root $s_5$ of $P_{5_Z}$ given the asymptotic behavior of the $\I_Z$ equivariant $g_Z(w)$:
$$s_5=J_{5_Z}(w_\infty)\quad \text{where}\ w_{\infty}=\lim_{k\to \infty}{g_Z^k(w_0)}.$$
(We can take $w_{\infty}$ to be either point in the superattracting $2$-cycle.)

 As a vertex in the special soccer ball $\Ss_Z$ determined by $\I_Z$, $w_{\infty}$ (and it's $g_Z$-image) belongs to one of five ``tetrahedral icosahedra" $\tau_a,\ a=1,\dots,5$ into which $\Ss_Z$ decomposes. Each $\tau_a$ is a $12$-point \T{a} orbit.  Figure~\ref{fig:tet_icos} displays the vertices of these icosahedra each of which is referenced by a color.  (Of course, the configuration of the points in $\Ss_Z$ will have symmetries that depend on $Z$.)

\begin{figure}[ht]
\resizebox{4in}{!}{\includegraphics[]{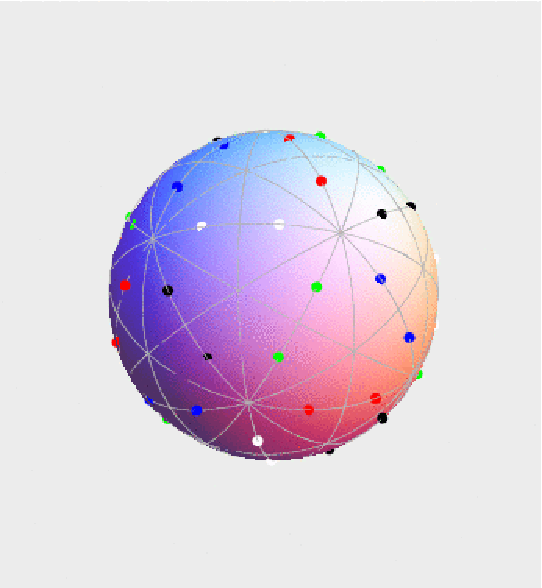}}
\caption{Five tetrahedral icosahedra in the soccer ball's vertices}
\label{fig:tet_icos}
\end{figure}

To obtain a second root of $P_{5_Z}$, we consider the \T{a}-symmetric quartic (for the relevant $a$)
$$P_{4_Z}=\frac{P_{5_Z}}{s-s_5}.$$
Now, in the \T{a} setting, build a root extractor $J_{4_Z}$ that relies on our having found a point in a $12$-point orbit and compute
$$s_4=J_{4_Z}(w_{\infty}).$$
As a practical matter, the specific $\tau_a$ to which $w_{\infty}$ belongs is
unknown.  However, we can determine the correct index by constructing
$J_{4_Z}$ for each $a$ and then evaluating $P_{4_Z}(J_{4_Z}(w_{\infty}))$.

The algorithm continues in a similar fashion by decomposing $\tau_a$ into four $3$-point \Z{3} orbits in order to obtain a solution $s_3=J_{3_Z}(w_{\infty})$ to the cubic
$$P_{3_Z}=\frac{P_{4_Z}}{s-s_4}.$$
At this stage, we have broken all of the $60$-fold symmetry and the roots of the final quadratic emerge.

In order for this cascade of solutions to flow, a single iteration
must produce a member of a $60$-point \I-orbit; any smaller orbit would not
provide a \Z{3} orbit with more than one element.  For instance, suppose the
dynamical output $w_\infty$ is a $20$-point and so, resides in one of the
five tetrahedral orbits of size $4$.  After extracting one of the quintic's
roots, solving the quartic factor that remains requires that we determine to
which of the four $1$-point \Z{3} orbits $w_\infty$ belongs.  But, having done
so, the \Z{3} orbit---being a single point---cannot distinguish the roots of the
cubic factor that divides the quartic.

\section{Acknowledgment}

Thanks go to Peter Doyle for a query and discussion that inspired this
work.

\end{document}